\documentclass{amsart}
\usepackage{amsmath}
\usepackage{amsfonts}
\usepackage{amssymb}
\usepackage{graphicx}

 \newtheorem{thm}{Theorem}[section]
 \newtheorem{cor}[thm]{Corollary}
 
 \newtheorem{prop}[thm]{Proposition}
 \theoremstyle{definition}

 \newtheorem{rem}[thm]{Remark}
 
 \numberwithin{equation}{section}

\begin{document}

\title[Elements with $r$--th roots in finite groups ]
 {Elements with $r$--th roots in finite groups }


\author[E. Khamseh]{E. Khamseh}
\address{Department of Mathematics, Islamic Azad University,
Mashhad--Branch, 91756, Mashhad, Iran}
\email{elahehkhamseh@gmail.com}

\author[M. R. R. Moghaddam]{M. R. R. Moghaddam}
\address{Department of Mathematics,
Khayyam Higher Education Institute, and
Centre of Excellence in Analysis and Algebraic Structure of
Ferdowsi University of Mashhad,
Mashhad, Iran}
\email{mrrm5@yahoo.ca}

\author[F. G. Russo]{F. G. Russo}
\address{DIEETCAM, University of Palermo, Viale delle Scienze, 90128, Palermo, Italy, and,
Department of Mathematics, Universiti Teknologi Malaysia, 81310, Skudai, Johor Bahru, Malaysia
}
\email{francescog.russo@yahoo.com}

\author[F. Saeedi]{F. Saeedi}
\address{Department of Mathematics, Islamic Azad University, Mashhad--Branch, 91756, Mashhad, Iran} \email{saeedi@mshdiau.ac.ir}


\subjclass[2010]{Primary 20D15; 20P05; Secondary  20D60.}

\keywords{$r$--th roots, probability, equations over finite groups, linear groups.}

\dedicatory{}

\begin{abstract}
The probability that a randomly chosen element of a finite group is an $r$--th root (for any integer $r\geq2$) has been studied largely in case $r=2$. Certain techniques  may be generalized for $r>2$ and here we find the exact value of this probability for  projective  special linear groups.  A result of density is placed at the end, in order to show an analogy with  the case $r=2$.  
\end{abstract}

\maketitle

\section{Introduction}

In the present paper all the groups are finite. In a group $G$, if there exists an element $y \in G$ for which $x = y^r$,  we say that $x$ $has$ $an$ $r$--$th$
$root$. For $r=2$, J. Blum described in \cite{blum} the probability
\[\mathrm{Prob}_2(S_n)=\frac{|S^2_n|}{n!}\]
that a randomly chosen permutation of length $n$ has a $2$--nd root (or square
root), where $S_n$ is the permutation group on $n$ letters. Successively his work was generalized in
\cite{d,lp1,lp2,russo} to the case of an arbitrary group. Already in \cite{bmw,p} it was studied the probability
\[\mathrm{Prob}_r(S_n)=\frac{|S^r_n|}{n!}\]
that a randomly chosen permutation of length $n$ has an $r$--th root for $r\geq2$. Therefore, many results in
\cite{blum} can be found as special situations of \cite{bmw,p}, but, so far as we have searched in the
literature, \cite{bmw,p} have not been extended in the sense of \cite{d,lp1,lp2,russo} to the case of an arbitrary
group. This is the beginning of our investigations and the motivation of the present work. We define the probability
\[\mathrm{Prob}_r(G)=\frac{|G^r|}{|G|},\]
where $r\geq2$ and $G^r=\{g^r \ | \ g\in G\}$ is the set of all elements of $G$ having at least one $r$--th
root. Unfortunately, $G^r$ is not a subgroup of $G$ but only a set and this can give difficulties from the
general point of view.

Even if $\mathrm{Prob}_2(G)$ is known by \cite{d,lp1,lp2} and $\mathrm{Prob}_r(S_n)$ by \cite{bmw,p}, we have not found whether it is
possible to obtain some structural information on $G$ from the bounds of $\mathrm{Prob}_r(G)$ or not. In the present paper we will investigate such
aspects and provide some restrictions of numerical nature for $\mathrm{Prob}_r(G)$.

\section{Basic properties}

We recall some fundamental notions on abelian groups. If $A$ is an abelian group, then     \[A^r=\{a^r \ |\ a \in A\}\] is a subgroup of $A$ and $A$ is called \textit{r-divisible}, if $A^r=A$. $A$ is \textit{divisible}, if it is $r$--divisible for all $r\geq2$. It is easy to see that $A$ is divisible if, and only if, it is $p$--divisible for all primes $p$. \[A[r]=\{a\in A \ |  \ a^r=1\}\] is a subgroup of $A$ and $A$ is  of \textit{exponent} $r$, if $A[r]=A$. If $r=p$ is a prime, $A[p]$ is called  \textit{p-socle} of $A$ and is isomorphic to the additive group of a vector space over the field with $p$ elements:  In other words,  $A[p]$ is an elementary $p$--group of rank $k\geq1$, that is, \[A[p]= C_p \times \ldots \times C_p=C^k_p.\]   In general, for an arbitrary $r\ge2$, the subgroups $A^r$ and $A[r]$ are related by the First Isomorphism's Theorem: $\varphi: a \in A \mapsto a^r\in A^r$ is a homomorphism of groups, inducing  $A^r\simeq A/A[r].$ The following remark gives a complete characterization  for abelian groups.

\begin{rem}\label{r:1}
Assume that $G$ is a nontrivial abelian group.
\begin{itemize}
\item[(i)] $\mathrm{Prob}_r(G)=\frac{1}{|G[r]|}$.
\item[(ii)]  If $r$ is prime, then $\mathrm{Prob}_r(G)=\frac{1}{r^k}$ for some $k\geq1$. Furthermore, the set 
\[X=\{\mathrm{Prob}_r(G) \ | \ G \mathrm{\ is \
an \ abelian \ group} \}\]
coincides with the subset 
\[Y=\left\{\frac{1}{r^k} \
\Big| \ k\geq 1 \right\}\]
of the interval $[0,1]$.
\end{itemize}
\end{rem}

From Remark \ref{r:1} (i), a nontrivial abelian group $G$ of exponent $r$ has $\mathrm{Prob}_r(G)=\frac{1}{|G|}$. Now we will  summarize most of the above considerations in the next result.

\begin{prop}\label{p:1}
Let $G$ be a nontrivial abelian group.
\begin{itemize}
\item[(i)]  $\mathrm{Prob}_r(G) = \frac{1}{|G[r]|}$. Furthermore, if $r$ is prime,  $\mathrm{Prob}_r(G)=\frac{1}{|G|}$ if and only if $G\simeq C^k_r$ for some $k\geq 1$.
\item[(ii)] The sets $X$  and $Y$ of Remark \ref{r:1} coincide.
\end{itemize}
\end{prop}

\begin{proof} (i). The first part is exactly Remark \ref{r:1} (i). Now assume that $r$ is prime. If $G\simeq C^k_r$, then $G$ is isomorphic to the additive group of a vector space over the field with $p$ elements, that is,  $G[r]\simeq G$. Then  $\mathrm{Prob}_r(G)=\frac{1}{|G|}$.  Conversely, $\mathrm{Prob}_r(G)=\frac{|rG|}{|G|}=\frac{|G|}{|G[r]| \ |G|}=\frac{1}{|G[r]|}=\frac{1}{|G|}$ implies $|G[r]|=|G|$, then $1=|rG|= \frac{|G|}{|G[r]|}$ via the isomorphism induced by $\varphi$, and so $G[r]\simeq G$, from which the result follows.

(ii). It is exactly Remark \ref{r:1} (ii).
\end{proof}

Now we  describe   $\mathrm{Prob}_r(G)=1$ for $r\ge 2$ and recall notions in \cite{lp1,lp2}.

\begin{rem} \label{r:2}
 Let $G$ be a nontrivial group.
\begin{itemize}
\item[(i)] $\mathrm{Prob}_r(G)=1$ if and only if $|G|=|G^r|$, that is, the number of the elements of $G$ having an $r$--th root is the same of the number of the elements of
$G$. 
\item[(ii)](See \cite{lp2}) The number of solutions of the equation $x^r = a$ in $G$ is
a multiple of $\mathrm{gcd}(r, |C_G(a)|)$, where $r\geq 2$ and $a$, $x\in G$. In particular, the number of solutions of the equation $x^r = 1$ over $G$ is a multiple of $\mathrm{gcd}(r, |G|)$ and when $r$ is prime, $|x|=|a|$ or $|x|=r|a|$.
\end{itemize}
\end{rem}
We reformulate Remark \ref{r:2}  as follows.
\begin{prop}\label{p:2}
Let $G$ be an arbitrary group and $r\ge2$.  $\mathrm{Prob}_r(G)=1$ if and only if
some multiple of $gcd(r, |G|)$ is equal to $1$.
\end{prop}

Propositions \ref{p:1} and \ref{p:2} agree with \cite[Proposition 2.1]{lp1}, when $r=2$.
Now we will proceed to list further properties.

\begin{rem}\label{r:5}
For an arbitrary  group $G$, we have $0<\frac{1}{|G|}\leq
\mathrm{Prob}_r(G)\leq 1$. Propositions \ref{p:1} (i) shows a condition in which we achieve the
lower bound $\frac{1}{|G|}$ in the abelian case. Proposition \ref{p:2} shows a more general condition in which we achieve the upper bound.
\end{rem}

 It is well--known that the probability of independent events is multiplicative. Here we have as follows.

\begin{prop}\label{l:1} Given two groups $A$ and $B$, $\mathrm{Prob}_r(A \times B)= \mathrm{Prob}_r(A) \  \mathrm{Prob}_r(B).$\end{prop}

\begin{proof} 
\[\mathrm{Prob}_r(A\times B)=\frac{|(A\times B)^r|}{|A\times B|}=\frac{|A^r\times B^r|}{|A||B|}=\frac{|A^r| |B^r|}{|A||B|}=\mathrm{Prob}_r(A) \ \mathrm{Prob}_r(B).\]
\end{proof}

For products of groups we draw the following conclusion.

\begin{prop}\label{l:2} If $G=AB$, where $A$ and $B$ are subgroups of $G$ such that $[A,B]=1$, then \[\mathrm{Prob}_r(G)= \frac{1}{|A^r \cap B^r|}\mathrm{Prob}_r(A) \mathrm{Prob}_r(B).\] In particular, if $A \cap B=1$, then $\mathrm{Prob}_r(G)=\mathrm{Prob}_r(A)\mathrm{Prob}_r(B)$.\end{prop}

\begin{proof} Given $a\in A$ and $b\in B$, $(ab)^r=a^rb^r$ if and only if $[a,b]=1$. Therefore $[A,B]=1$ implies $(AB)^r=A^rB^r$ and so $|A^rB^r|=\frac{|A^r||B^r|}{|A^r \cap B^r|}$. Then
\[\mathrm{Prob}_r(G)=\frac{|G^r|}{|G|}=\frac{|(AB)^r|}{|AB|}=\frac{|A^rB^r|}{|AB|}=\frac{1}{|A^r\cap B^r|}\frac{|A^r|}{|A|}\frac{|B^r|}{|B|}\]\[=\frac{\mathrm{Prob}_r(A)
\mathrm{Prob}_r(B)}{|A^r \cap B^r|}.\]
In particular, $A^r \cap B^r \subseteq A \cap B=1$ implies
$\mathrm{Prob}_r(G)=\mathrm{Prob}_r(A)\mathrm{Prob}_r(B)$.
\end{proof}

The next two results show bounds in terms of subgroups and quotients.

\begin{prop}\label{l:3} Let $N$ be a normal subgroup of a group $G$. Then \[\mathrm{Prob}_r(G)
\leq \mathrm{Prob}_r(G/N).\] 
\end{prop}

\begin{proof}Note that $gN \in G/N$ has an $r$--th root if and only if there is $xN
\in G/N$ for which $gN = (xN)^r$, that is, $x^r\in gN$. Therefore $gN \in G/N$ does not have an $r$--th root if
and only if there is no element $x \in G$ with $x^r\in  gN$. Hence, if a coset in $G/N$ does not have an $r$--th
root, then no element of this coset has an $r$--th root in $G$, and therefore $|G| - |G^r|\geq |N| (|G/N| -
|(G/N)^r| )$ . By dividing both sides by $|G|$ we obtain $1-\mathrm{Prob}_r(G) \geq 1-\mathrm{Prob}_r(G/N)$ and so $\mathrm{Prob}_r(G) \leq
\mathrm{Prob}_r(G/N)$, as required. 
\end{proof}

\begin{prop}\label{l:5} Let $H$ be a subgroup of a group $G$. Then \[|G|^{-1}\mathrm{Prob}_r(H)\leq \mathrm{Prob}_r(G).\]
\end{prop}

\begin{proof}Obviously $H^r \subseteq G^r$ implies $|H^r|\leq |G^r|$. Therefore  $\mathrm{Prob}_r(H) \leq |H| \cdot \mathrm{Prob}_r(H)=\frac{|H|}{|H|} \cdot |H^r|\leq \frac{|G|}{|G|} \cdot
|G^r|=|G| \cdot \mathrm{Prob}_r(G)$ implies $|G|^{-1}\mathrm{Prob}_r(H)\leq \mathrm{Prob}_r(G) $ and the lower bound follows.
\end{proof}

The following result is  a lower bound of general interest.

\begin{cor}\label{c:7}
Let $G$ be a solvable group and  $P$ be a Sylow $p$--subgroup of $G$ for some prime $p$. Then $\frac{1}{|P|}\leq \mathrm{Prob}_p(G)$. 
\end{cor}

\begin{proof} 
Since $H$ is solvable, there exists a $p'$--Hall subgroup $H$ of $G$ such that $|G| = |H||P|$ and $H = H^p
\subseteq G^p$. Therefore, $\mathrm{Prob}_p(G) = \frac{|G^p|} {|G|} \geq \frac{|H|}{ |G|} = \frac{|H|}{|H||P|} = \frac{1}{ |P|}.$ \end{proof}

\section{Projective special linear groups and density}

\begin{thm}\label{psl}Let  $q$ be a prime power. If $q \equiv 1 \mod 4$, $q$ is odd and $r=\frac{q-1 }{2} \ge 2$ is  prime, then
\[\mathrm{Prob}_r(\mathrm{PSL}(2,q))=\frac{r+1}{2r}.\]
In particular, if $r=2$, then $\mathrm{Prob}_2(\mathrm{PSL}(2,q))=\frac{3}{4}.$
\end{thm}
\begin{proof}
We recall that \[|\mathrm{PSL}(2,q)| = \frac{q(q-1)(q + 1)}{\gcd (2, q-1)}=q(q-1)(q + 1).\] Let $\nu$ be  a generator of the multiplicative group of the field of $q$ elements. Denote
\begin{displaymath}\label{displ:1}
1=\left(
\begin{array}{cccccccc}
1 & 0  \\
0 & 1 \\
\end{array}\right), \ 
c=\left(
\begin{array}{cccccccc}
1 & 0 \\
1 & 1 \\
\end{array}\right), \ 
d=\left(
\begin{array}{cccccccc}
1 & 0 \\
\nu & 1 \\
\end{array}\right), \
a=\left(
\begin{array}{cccccccc}
\nu & 0  \\
0 & \nu^{-1}\\
\end{array}\right).
\end{displaymath}
and $b$ an element of order $q + 1$ (Singer cycle) in $\mathrm{SL}(2,q)$. By abuse of notation, we use the same symbols for the corresponding elements in $\mathrm{PSL}(2,q)= \mathrm{SL}(2,q)/Z(\mathrm{SL}(2,q))$. From the character
table of $\mathrm{SL}(2,q)$ (see \cite[Theorem 38.1]{dorn}), one gets easily the character table of $\mathrm{PSL}(2,q)$. We reproduce it below for the convenience of the reader.
The elements $1, c, d, a^l$  and $b^m$ for $1 \le l \le \frac{q-1}{4}$ and $1 \le m \le \frac{q-1}{4}$ form a set of representatives for the conjugacy classes of $\mathrm{PSL}(2,q)$. For $1 \le l \le \frac{q-1}{4}$ and $1 \le m \le \frac{q-1}{4}$,  one can see from \cite[Theorem 38.1]{dorn} that \[|C_{\mathrm{PSL}(2,q)}(1)|=|G|, \ \ |C_{\mathrm{PSL}(2,q)}(c)|= |C_{\mathrm{PSL}(2,q)}(d)|=p,\]
\[  |C_{\mathrm{PSL}(2,q)}(a^l)|=\frac{q-1}{2}, \ \ |C_{\mathrm{PSL}(2,q)}(a^{\frac{q-1}{4}})|=q-1,  \ \ |C_{\mathrm{PSL}(2,q)}(b^m)|=\frac{q+1}{2}, \]
where $p$ is the prime of which $q$ is  power. Now we count the elements which do not have $r$--th roots and will deduce the probability of having $r$--th roots.

Since $\langle a \rangle \simeq C_r$, $\langle a\rangle [r] \simeq \langle a\rangle$ and $\langle a\rangle^r=1$,  Proposition \ref{p:1} (i) implies $\mathrm{Prob}_r(\langle a \rangle)= \frac{1}{r}$ so the elements not having $r$--th roots in $\langle a\rangle$  are exactly \[|\langle a \rangle - \langle a\rangle^r|= |\langle a \rangle| - |\langle a\rangle^r|= |\langle a\rangle|-1=r-1= \frac{q-1}{2}-1=\frac{q-3}{2}.\]

On the other hand, the (distinct) conjugates of $\langle a \rangle$ have trivial intersection
with $\langle a \rangle$ so that the total number of elements of $\mathrm{PSL}(2,q)$ which do not have $r$--th roots is obtained by multiplying $\frac{q-3}{2}$ by the number of conjugates of $\langle a \rangle$ , which is $|\mathrm{PSL}(2,q) : N_{\mathrm{PSL}(2,q)}(\langle a \rangle )|$. This means that
\[|\mathrm{PSL}(2,q)|-|\mathrm{PSL}(2,q)^r|=|\mathrm{PSL}(2,q) - \mathrm{PSL}(2,q)^r|\]
\[= |\langle a \rangle - \langle a\rangle^r| \cdot |\mathrm{PSL}(2,q) : N_{\mathrm{PSL}(2,q)} ( \langle a \rangle )|
=\frac{q-3}{2} \cdot \frac{|\mathrm{PSL}(2,q)|}{|N_{\mathrm{PSL}(2,q)} (\langle a \rangle) |}\]
\[= \frac{q-3}{2} \cdot \frac{|\mathrm{PSL}(2,q)|}{q-1}\]
and, dividing both sides by $|\mathrm{PSL}(2,q)|$, we get
\[1-\mathrm{Prob}_r(\mathrm{PSL}(2,q))= \frac{q-3}{2} \cdot \frac{1}{q-1},\]
that is,
\[\mathrm{Prob}_r(\mathrm{PSL}(2,q))=1-\frac{q-3}{2(q-1)}=\frac{q+1}{2(q-1)}=\frac{2(r+1)}{2(2r)}= \frac{r+1}{2r}.\]
\end{proof}

We note that the case $r=2$, which appears in the previous theorem, was found in \cite[Proposition 3.1]{lp1}. A consequence is the following.

\begin{cor}Let  $q$ be a prime power. If $q \equiv 1 \mod 4$, $q$ is odd and $r=\frac{q-1 }{2} \ge 2$ is  prime, then $\lim_{r \rightarrow \infty}\mathrm{Prob}_r(\mathrm{PSL}(2,q))= \frac{1}{2}.$
\end{cor}

The computations for the projective special linear groups are important in order to get \cite[Theorem 1.1]{d} and to prove that the set  
\[Z=\{\mathrm{Prob}_2(G) \ | \  G \ \mathrm{is} \ \mathrm{an} \ \mathrm{arbitrary} \ \mathrm{group}\}\] is dense in [0,1].  We are going to generalize  for any prime $r\geq2$, and we will not use projective special linear groups as done in \cite[Theorem 1.1]{d}, but will assume a priori the existence of a certain group with a prescribed value of probability. This is justified by evidences of computational nature.

\begin{cor}\label{c:4} For any $\epsilon \in \mathbb{R}$ with $\epsilon
> 0$ and given a prime $r\ge2$, there exists an abelian group $A$ such that
$0 < \mathrm{Prob}_r(A) < \epsilon$.\end{cor}

\begin{proof}Let $k>1$ be such that $1/r^k < \epsilon$ and  $A$ be an elementary
 $r$--group  of rank $k$. By Proposition \ref{p:1} (ii), the result follows. \end{proof}

A proof of Corollary \ref{c:4} when $r=2$ can be found in \cite{d}. Briefly, Corollary \ref{c:4} shows that $0$ is an accumulation point for the set  $X$ in Remark \ref{r:1}.

\begin{cor}\label{c:4extra} Assume that $r \ge 2$ is a prime and $S$ is a  group such that $\mathrm{Prob}_r(S)=1-\frac{1}{|R|}$ for an elementary abelian $r$--Sylow subgroup $R$ of $S$. Then for any $\epsilon \in \mathbb{R}$ with $\epsilon > 0$ we have $1-\epsilon < \mathrm{Prob}_r(S) < 1$.\end{cor}

\begin{proof}  Since $R$ is a Sylow $r$--subgroup of $S$ which is elementary abelian of rank $k$ for some $k\ge1$, we have $1/r^k < \epsilon$ and 
$\mathrm{Prob}_r(S) = 1-\frac{1}{r^k}=\frac{r^k-1}{r^k}$. On the other
hand, $1 - \epsilon < \frac{r^k-1}{r^k} < 1$, therefore $1 -
\epsilon < \mathrm{Prob}_r(S) < 1$, as claimed. \end{proof}

 Corollary \ref{c:4extra} when $r=2$ can be found in \cite{d}. Also Corollary \ref{c:4} is illustrating that $1$ is an accumulation point for the set \[T=\{\mathrm{Prob}_r(G) \ | \  G \ \mathrm{is} \ \mathrm{an} \ \mathrm{arbitrary} \ \mathrm{group}\},\]
where $r\ge2$ is a given prime.

\begin{thm}\label{t:density} Let $r\ge2$ be a prime and assume that there exists a group $H$  such that $\mathrm{Prob}_r(H)=1-\frac{1}{|R|}$ for an elementary abelian $r$--Sylow subgroup $R$ of $H$. Then the set  $T$  is dense in $[0,1]$.\end{thm}

\begin{proof} By Corollaries \ref{c:4},
\ref{c:4extra}, there is no loss of generality in showing that, if $0< x
<1$, then $x$ is a limit point of $T$. There exists an integer $m$
such that $1/r< r^mx<1$. Note that $(0,1)= {\underset {m\geq 0}
\bigcup} [1/r^{m+1}, 1/r^m)$. Let $y=r^mx$. We can choose an integer
$n_1 \geq1$ such that \[(r^{n_1}-1)/r^{n_1}\leq y \leq
(r^{n_1+1}-1)/r^{n_1+1},\] noting that
$[1/r,1)={\underset {n\geq 1} \bigcup} [(r^n-1)/r^n,
(r^{n+1}-1)/r^{n+1})$. Let $s_1=(r^{n_1}-1)/r^{n_1}$ and
$r_1=(r^{n_1+1}-1)/r^{n_1+1}$. Again we can choose an integer $n_2
\geq1$ such that \[(r^{n_2}-1)/r^{n_2}\leq y/r_1 \leq
(r^{n_2+1}-1)/r^{n_2+1},\] noting that $1/r \leq y/r_1 <
1 $. As before, let $s_2=(r^{n_2}-1)/r^{n_2}$ and
$r_2=(r^{n_2+1}-1)/r^{n_2+1}$. Iterating this process, there exist
positive integers $n_1, n_2, n_3, \ldots$ and two sequences
$\{s_i\}$ and $\{r_i\}$ such that $s_i=(r^{n_i}-1)/r^{n_i},
 r_i=(r^{n_i+1}-1)/r^{n_i+1}$ and  $s_i \leq \frac{y}{r_1r_2 \ldots r_{i-1}} < r_i$ for all $i\geq 1$.
Of course, $0<s_i<r_i<1$ for all $i\geq1$. We have $n_i \leq
n_{i+1}$ for all $i\geq1$, since \[s_i \leq
\frac{y}{r_1r_2 \ldots r_{i-1}} < \frac{y}{r_1r_2 \ldots
r_{i-1}r_i}<r_{i+1}. \] Thus $\{s_i\}$ is a
monotonically increasing sequence, bounded by 1, and so convergent.
Moreover, $\{s_i\}$ has infinitely many distinct terms; otherwise
$\{s_i\}$, and hence $\{r_i\}$, would be eventually constant, and
so, for some $j\geq 1$, we would have
\[\frac{y}{r_1r_2 \ldots r_{j-1} r^{k-1}_j} < r_j\] or
$r_1r_2 \ldots  r_{j-1} r^k_j$ for $k\geq 1$. This is impossible,
since $y>0$  and ${\underset{k\rightarrow \infty}\lim}r^k_j=0$.
Therefore, $\{s_i\}$ converges to 1 (after omitting repeated terms),
because it is a subsequence of $\{(r^n-1)/r^n\}$. This allows us to
note that the sequence $\{a_i\}$ converges to 1, where $a_i=y/r_1r_2
\ldots r_{i-1}$. Consequently, the sequence $\{b_i\}$ converges to
$y$, where $b_i=r_1r_2 \ldots r_{i-1}$. Thus we have
\[\lim_{k\rightarrow \infty}\frac{r_1r_2 \ldots
r_{i-1}}{r^m}=\frac{y}{r^m}=x.\] For each $i\geq1$ we
consider the  group $G^{(i)}=G_0 \times G_1 \times \ldots
\times G_{i-1}$, where 
$G_0=C^m_r$ and $G_k$ is a sequence of groups isomorphic for each $k$ to the group $H$, introduced in the assumptions. Propositions \ref{p:1} and  \ref{l:1}  imply
\[\mathrm{Prob}_r(G^{(i)})=\mathrm{Prob}_r(G_0) \ \mathrm{Prob}_r(G_1) \ldots \mathrm{Prob}_r(G_{i-1})=\frac{1}{r^m}r_1r_2
\ldots r_{i-1}.\] We have ${\underset{i\rightarrow
\infty}\lim}\mathrm{Prob}_r(G^{(i)})=x$ and  the result follows.
\end{proof}


\begin{thebibliography}{20}

\bibitem{blum} J. Blum, Enumeration of the square permutations in $S_n$,  {\it J. Comb. Theory Ser. A} {\bf 17} (1974), 156-161.


\bibitem{bmw} M. B${\rm \acute{o}}$na, A. McLennan and D. White, Permutations with roots, {\it Random structures and algorithms} {\bf 17} (2) (2000), 157--167.

\bibitem{d} A.K. Das, On group elements having square roots, {\it Bull. Iranian Math. Soc.} {\bf 31} (2005), 33--36.





\bibitem{dorn}L. Dornhoff, \textit{Group Representation Theory}, Part A, Marcel Dekker, New York, 1971.











\bibitem{lp1} M.S. Lucido and M.R. Pournaki, Elements with square roots in finite groups, {\it Algebra Colloq.} { \bf 12} (2005), 677--690.

\bibitem{lp2} M.S. Lucido and M.R. Pournaki, Probability that an element of a finite group has a square root, \textit{Colloq. Math.} \textbf{112} (2008), 147--155.


\bibitem{p}N. Pouyanne, On the number of permutations admitting an m-th root, {\it Electr. J. Comb.} {\bf 9} (2002), 12 pp. (electronic).


\bibitem{russo}F.G. Russo, Elements with square roots in compact groups, \textit{Asian Eur. J. Math.} \textbf{3} (2010), 495--500.

\end{thebibliography}
\end{document}